\documentclass[english, 11pt]{amsart}

\usepackage{tikz}
\usepackage{aliascnt}
\usepackage{mathrsfs}
\usepackage[all,poly,knot]{xy}
\usepackage{hyperref}
\usepackage{comment}  
\usepackage{csquotes}   
\usepackage{amssymb,amsbsy,amsmath,amsfonts,amssymb,amscd,
	graphics,color,footmisc,fancyhdr,multicol,fancybox,
	graphicx,mathrsfs,rotating,ifthen,wasysym}

\usepackage{times}
\usepackage{pdfpages}
\usepackage{multirow}
\usepackage{nccrules}
\usepackage{textcomp}
\usepackage{comment}
\usepackage{framed}
\usepackage[all]{xy}
\usepackage{graphicx}
\usepackage{pgf,tikz}
\usepackage{mathrsfs}
\usetikzlibrary{arrows}
\usepackage{lipsum}
\usepackage{mathtools}

\def\log{\mathrm{log}\,}

\theoremstyle{plain}

\newtheorem{thm}{Theorem}[section]  

\newtheorem{cor}[thm]{{Corollary}} 

\newtheorem{lem}[thm]{{Lemma}}

\theoremstyle{remark}

\numberwithin{equation}{section}

\usepackage[top=1.5in, bottom=1.5in, left=1in, right=1in]{geometry}

 \usepackage[hyperpageref]{backref}
 
\usepackage{xcolor}
\hypersetup{
	colorlinks,
	linkcolor={red!50!black},
	citecolor={blue!62!black},
	urlcolor={blue!80!black}
}
\theoremstyle{plain}
\newcommand{\thistheoremname}{}
\newtheorem*{genericthm*}{\thistheoremname}
\newenvironment{namedthm*}[1]{\renewcommand{\thistheoremname}{#1}%
	\begin{genericthm*}}
	{\end{genericthm*}}

\newtheoremstyle{named}{}{}{\itshape}{}{\bfseries}{.}{.5em}{\thmnote{#3's }#1}
\theoremstyle{named}

\makeatletter
\newcommand\thankssymb[1]{\textsuperscript{\@fnsymbol{#1}}}
\makeatother
\begin{document} 
\title[A $p$-adic Second Main Theorem]{\bf A $p$-adic Second Main Theorem} 

%	\author{Ruiran Sun}
%		 \address{Institut fur Mathematik, Universit\"at Mainz, Mainz, 55099, Germany}
%	\email{ruirasun@uni-mainz.de}


\author{Dinh Tuan Huynh}

\address{Department of Mathematics, University of Education, Hue University, 34 Le Loi St., Hue City, Vietnam}
\email{huynhdinhtuan@dhsphue.edu.vn}

\dedicatory{Tien Canh Nguyen, in memoriam}

\begin{abstract}
Let $\mathbf{K}$ be an algebraically closed field of arbitrary characteristic, complete with respect to a non-archimedean absolute value $|\,|$. We establish a Second Main Theorem type estimate for analytic map $f\colon \mathbf{K}\rightarrow\mathbb{P}^n(\mathbf{K})$ and a family of $n$ hypersurfaces in $\mathbb{P}^n(\mathbf{K})$ intersecting transversally and not all being hyperplanes. This implements the previous work of Levin where the case of all hypersurfaces having degree greater than one was studied.
\end{abstract}

\maketitle
%\tableofcontents 

\section{Introduction }
The classical Nevanlinna Theory was established about one hundred years ago by comparing the frequency of impacts of a holomorphic map $f\colon\mathbb{C}\rightarrow\mathbb{P}^1(\mathbb{C})$ with $q\geq 3$ points in $\mathbb{P}^1(\mathbb{C})$ with the growth order of $f$. This theory extends  the fundamental theorem of algebra
from polynomials to meromorphic functions and could be  also regarded as a quantitative version of the Little Picard Theorem. Nevanlinna theory consists of two fundamental inequalities, called the First Main Theorem and the Second Main Theorem. The First Main Theorem is a direct consequence of the Poisson-Jensen formula, and we now have a satisfactory theory for it in any higher dimension. The deeper Second Main Theorem was established in several situations, but remains largely open in general.

The $p$-adic Nevanlinna theory studies the distribution of $p$-adic meromorphic functions on $\mathbb{C}_p$, the completion of the algebraic closure of the $p$-adic numbers $\mathbb{Q}_p$, in an analogous way the classical Nevanlinna theory studies complex meromorphic functions. More generally, this theory could be extended from $\mathbb{C}_p$ to any algebraically closed field $\mathbf{K}$ of arbitrary characteristic, complete with respect to a non-archimedean absolute value $|\cdot\,|$.

For an entire function $\varphi=\sum_{i=0}^{\infty}a_iz^i$ on the field $\mathbf{K}$, for each real number $r>0$, define
\begin{align*}
|\varphi|_r:&=\sup_i|a_i||r^i|\\
&=\sup\{|\varphi(z)|:z\in\mathbf{K},\, |z|=r\}\\
&=\sup\{|\varphi(z)|:z\in\mathbf{K},\, |z|\leq r\}.
\end{align*}
Let $f\colon\mathbf{K}\rightarrow\mathbb{P}^n(\mathbf{K})$ be a nonconstant analytic map and let $\tilde{f}=(f_0,\dots,f_n)$ be a reduced representation of $f$, where $f_0,\dots,f_n$ are entire functions on $\mathbf{K}$ without common zeros. The {\it order function} of $f$ is defined by
\[
T_f(r):=\log\|f\|_r,
\]
where $\|f\|_r=\max\{|f_0|_r,\dots,|f_n|_r\}$. This function is well defined, independent of the choice of the reduce representation of $f$ up to an additive constant.

Let $D$ be a hypersurface in $\mathbb{P}^n(\mathbf{K})$ defined by a homogeneous polynomial $Q\in\mathbf{K}[x_0,\dots,x_n]$ of degree $d\geq 1$. Suppose that the image of $f$ is not contained in $D$, then the entire function $Q\circ f$ is not identically zero. Let $n_f(r,D)$ be the number of zeros of $Q\circ f$ on the disc $\mathbf{B}_r:=\{z\in \mathbf{K}:|z|\leq r\}$, which is finite, the {\it counting function} of $f$ with respect to $D$ is then defined as
\[
N_f(r,D):=\int_{0}^{r}\dfrac{n_f(t,D)-n_f(0,D)}{t}dt+n_f(0,D)\,\log r,
\]
capturing the asymptotic frequency of intersection of $f(\mathbf{K})$ and $D$. The {\it proximity function} of $f$ associated to $D$ is defined as
\[
m_f(r,D):=\log \,\dfrac{\|f\|_r^d}{|Q\circ f|_r},
\]
which is independent of the choice of $Q$, up to an additive constant.

The Nevanlinna theory in the non-archimedean setting is then established by comparing the above three associated functions. Here and throughout this note, the implied constant in the term $O(1)$ is independent of $r$. The First Main Theorem in this context  reads as follows.
\begin{namedthm*}{First Main Theorem}
	Let $D$ be a hypersurface of degree $d$ in $\mathbb{P}^n(\mathbf{K})$. Let $f\colon\mathbf{K}\rightarrow\mathbb{P}^n(\mathbf{K})$ be a nonconstant analytic map with $f(\mathbf{K})\not\subset D$. Then for any $r>0$, one has
	\[
	m_f(r,D)+N_f(r,D)=d\, T_f(r)+O(1),
	\]
	and consequently
	\[
	N_f(r,D)\leq d\,T_f(r)+O(1).
	\]
\end{namedthm*} 

On the other hand, in the harder part, one tries to find an upper bound for $T_f(r)$ in terms of the sum of certain counting functions. Such types of estimates are so-called Second Main Theorems.

A family of hypersurfaces $\{D_i\}_{1\leq i\leq q}$ in $\mathbb{P}^n(\mathbf{K})$ is said to be in general position if $\dim\cap_{i\in I}D_i\leq n-|I|$ for every subset $I\subset\{1,\dots,q\}$ having cardinality $|I|\leq n+1$, with the convention that $\dim\varnothing=-1$.  Unlike the case of 
Nevanlinna theory in complex setting (cf. \cite{Ru04,Ru09}), Ru~\cite{Ru2001} pointed out that in the $p$-adic context, if one disregards the ramification term, then the Second Main Theorem  could be deduced directly from the First Main Theorem. Based on this observation, he  gave a simple proof for Second Main Theorem for curvilinear hypersurfaces $D_i$ in general position in projective space.
\begin{namedthm*}{Ru's Second Main Theorem}
	Let $\{D_i\}_{1\leq i\leq q}$ be a family of hypersurfaces in general position in $\mathbb{P}^n(\mathbf{K})$. Let $f\colon\mathbf{K}\rightarrow\mathbb{P}^n(\mathbf{K})$ be a nonconstant analytic map whose image is not completely contained in any of the hypersurfaces $D_1,\dots,D_q$. Then for any $r>0$, one has
	\[
	\sum_{i=1}^q\dfrac{m_f(r,D_i)}{\deg(D_i)}
	\leq
	n\,T_f(r)+O(1).
	\]
\end{namedthm*}
In the case where all $D_i$ are hyperplanes, a Second Main Theorem involving the ramification term was established by Boutabaa~\cite{Boutabaa91} (see also \cite{Khoai-Tu95, Cherry-Ye1997}). An~\cite{An2007} extended Ru's result to the case of arbitrary projective variety.  

When all hypersurfaces $D_i$ have degree greater than one, under an extra generic geometric assumption about transversal intersections, Levin~\cite{Levin2015} established a Second Main Theorem type estimate involving the degree of $D_i$, which improved the above mentioned results of An, Ru. When the ambient space is $\mathbb{P}^n(\mathbf{K})$, Levin's result takes the following form.

\begin{namedthm*}{Levin's Second Main Theorem}
	Let $\{D_i\}_{1\leq i\leq q}$ be a family of smooth hypersurfaces in general position in $\mathbb{P}^n(\mathbf{K})$ such that all intersections among them are transversal. Let $f\colon\mathbf{K}\rightarrow\mathbb{P}^n(\mathbf{K})$ be a nonconstant analytic map whose image is not completely contained in any of the hypersurfaces $D_1,\dots,D_q$. Then for any $r>0$, one has
	\[
	\sum_{i=1}^q\dfrac{m_f(r,D_i)}{\deg(D_i)}
	\leq
	\bigg(n-1+\max_i\dfrac{1}{\deg (D_i)}\bigg)\,T_f(r)+O(1).
	\]
\end{namedthm*}

Our aim in this note is to treat the remaining case in the paper of Levin \cite{Levin2015},  where at least one of the $D_i$ is a hyperplane. Keeping the transversality assumption, we prove a Second Main Theorem type estimate for a family of $n$ hypersurfaces in $\mathbb{P}^n(\mathbf{K})$, not all being hyperplanes.
We shall employ the following fact extracted from the proof of the Second Main Theorem of Ru, An and Levin.
 \begin{lem}
 \label{only n proximity could be not bounded}
 	Let $\{D_i\}_{1\leq i\leq q}$ be a family of $q\geq n$ hypersurfaces in general position in $\mathbb{P}^n(\mathbf{K})$. If $f\colon\mathbf{K}\rightarrow\mathbb{P}^n(\mathbf{K})$ is a nonconstant analytic map whose image is not completely contained in any of the hypersurfaces $D_1,\dots,D_q$, then there exist at least $q-n$ indexes $i$ such that $m_f(r,D_i)=O(1)$.
 \end{lem}

\section{Results in projective plane}
\begin{thm}
\label{one line one conic}
	Let $f\colon \mathbf{K}\rightarrow\mathbb{P}^2(\mathbf{K})$ be a nonconstant analytic curve. Let $L$ be a line and $C$ be a nonsingular conic in $\mathbb{P}^2(\mathbf{K})$ and suppose that $L,C$ intersect transversally. If the image of $f$ is not completely contained in $L$ and $C$ and any tangent line of $\mathcal{C}$ at the intersection points $L\cap C$, then
	\[
m_f(r,L)+	\dfrac{m_f(r,C)}{2}\leq \dfrac{3}{2}\,T_f(r)+O(1).
	\]
\end{thm}
\begin{proof}
%If $m_f(r,C)\leq m_f(r,L)$, applying Levin \cite{Levin2015}, one has 
%\[
%m_f(r,L)+	\dfrac{m_f(r,C)}{2}\leq\dfrac{3}{2}\, T_f(r)+O(1).
%\]
%If $m_f(r,L)<m_f(r,C)$,
Let $H_1,H_2$ be the two tangent lines of $C$ at the intersection points of $L$ and $C$, then it follows from Max Noether's fundamental theorem that
\begin{equation}
\label{h1h2-ac+bl^2}
H_1H_2=aC+bL^2,
\end{equation}
for some nonzero constants $a,b$. Here we also use the notions $H_1,H_2,C,L$ to denote the homogeneous polynomials defining these curves. The equality \eqref{h1h2-ac+bl^2} yields
\[
|H_1\circ f|_r\cdot |H_2\circ f|_r\leq M\max\{|L\circ f|^2_r,|C\circ f|_r\},
\]
where $M$ is some positive constant independent of $r$. Hence either one has
\begin{equation}
\label{bound for mf L}
m_f(r,L^2)\leq m_f(r,H_1)+m_f(r,H_2)+O(1),
\end{equation}
or
\begin{equation}
\label{bound for mf C}
m_f(r,C)\leq m_f(r,H_1)+m_f(r,H_2)+O(1).
\end{equation}
If the inequality \eqref{bound for mf L} holds, then
\[
m_f(r,L)+	\dfrac{m_f(r,C)}{2}
\leq
\dfrac{m_f(r,H_1)+m_f(r,H_2)}{2}
+
\dfrac{m_f(r,C)}{2}
+O(1).
\]
By Lemma~\ref{only n proximity could be not bounded}, there is at least one bounded function among  $m_f(r,C), m_f(r,H_1), m_f(r,H_2)$. Therefore, the above inequality  yields
\[
m_f(r,L)+	\dfrac{m_f(r,C)}{2}
\leq
\dfrac{3}{2}\,T_f(r)+O(1).
\]
Similarly, if \eqref{bound for mf C} holds, then
\[
m_f(r,L)+	\dfrac{m_f(r,C)}{2}
\leq
m_f(r,L)+\dfrac{m_f(r,H_1)+m_f(r,H_2)}{2}+O(1).
\]
Again one can apply Lemma~\ref{only n proximity could be not bounded} to three proximity functions in the right hand side of this inequality and one receives
\[
m_f(r,L)+	\dfrac{m_f(r,C)}{2}
\leq
\dfrac{3}{2}\,T_f(r)+O(1).
\]
\end{proof}
\begin{cor}
Let $f\colon \mathbf{K}\rightarrow\mathbb{P}^2(\mathbf{K})$ be a nonconstant analytic curve. Let $L$ be a line and $C$ be a nonsingular conic in $\mathbb{P}^2(\mathbf{K})$ and suppose that $L,C$ intersect transversally. If $f$ avoids $L,C$, then the image of $f$ must be contained in some tangent line of $C$ at the intersection points of $L$ and $C$.	
\end{cor}
\begin{proof}
	Since $f$ avoids $L,C$, one has $m_f(r,L)=\dfrac{m_f(r,C)}{2}=T_f(r)+O(1)$.
Suppose on the contrary that the image of $f$ is not contained in any tangent line of $C$ at the intersection point of $L$ and $C$, then it follows from Theorem~\ref{one line one conic} that $2\, T_f(r)\leq\dfrac{3}{2}\, T_f(r)+O(1)$, a contradiction.
\end{proof}

When $C$ is a nonsingular plane curve of degree $d\geq 3$, we can put a generic assumption that the line $L$ does not pass through the point of multiplicity $d$ of $C$. Under this additional condition, the  assumption in Theorem~\ref{one line one conic} that the image of $f$ is not contained in any tangent line of $C$  at the intersection point $L\cap C$  can be removed.

\begin{thm}
	\label{one line one curve}
	Let $f\colon \mathbf{K}\rightarrow\mathbb{P}^2(\mathbf{K})$ be a nonconstant analytic curve. Let $L$ be a line and $C$ be a nonsingular curve of degree $d\geq 3$ in $\mathbb{P}^2(\mathbf{K})$. Suppose that $L,C$ intersect transversally and $L$ does not pass through the point of multiplicity $d$ of $C$. If  the image of $f$ is not completely contained in $L\cup C$, then
	\[
	m_f(r,L)+	\dfrac{m_f(r,C)}{d}\leq \bigg(2-\dfrac{1}{d}\bigg)\,T_f(r)+O(1).
	\]
\end{thm}
\begin{proof}
 Let $H_1,\dots, H_d$ be the  tangent lines of $C$ at the intersection points $A_1,\dots,A_d$ of $L$ and $C$, then
\[
H_1\cdots H_d=aC+BL^2,
\]
for some non zero constant $a$ and some homogeneous polynomial $B$ of degree $d-2$. This implies
\[
\dfrac{|H_1\circ f|_r\cdots |H_d\circ f|_r}{\|f\|_r^d}\leq M\,\max\bigg\{\dfrac{|L\circ f|^2_r}{\|f\|_r^2},\dfrac{|C\circ f|_r}{\|f\|_r^2} \bigg\},
\]
for some positive constant $M$ independent of $r$. Thus, if the image of $f$ is not completely contained in any $H_i$, then one of the following inequalities holds true:
\begin{equation}
\label{prox L<=sum of prox Hi}
m_f(r,L^2)\leq\sum_{i=1}^d m_f(r,H_i)+O(1)
\end{equation}
or
\begin{equation}
\label{prox C<=sum of prox Hi}
m_f(r,C)\leq\sum_{i=1}^d m_f(r,H_i)+O(1).
\end{equation}
Similarly as in the proof of Theorem~\ref{one line one conic}, in the case where \eqref{prox L<=sum of prox Hi} holds, one has
\begin{equation}
\label{non-degenerate case 1}
m_f(r,L)+	\dfrac{m_f(r,C)}{d}
\leq
\dfrac{\sum_{i=1}^dm_f(r,H_i)}{2}+\dfrac{m_f(r,C)}{d}+O(1)
\leq
\dfrac{3}{2}\,T_f(r)+O(1)
\leq
\bigg(2-\dfrac{1}{d}\bigg)\,T_f(r)+O(1).
\end{equation}
In the remain case where \eqref{prox C<=sum of prox Hi} holds, one has
\begin{align}
\label{non-degenerate case 2}
m_f(r,L)+	\dfrac{m_f(r,C)}{d}
&\leq
m_f(r,L)+
\dfrac{\sum_{i=1}^dm_f(r,H_i)}{d}+O(1)\notag\\
&\leq
\bigg(
1+
\dfrac{1}{d}
\bigg)\, T_f(r)+O(1)\notag\\
&\leq
\bigg(2-\dfrac{1}{d}\bigg)
\,
T_f(r)+O(1).
\end{align}
Consider the case where the image of $f$ is contained in some tangent line $H_i$, says $H_1$. Since $A_1$ is not the point of multiplicity $d$, the tangent line $H_1$ must intersect $C$ at other points $B_1,\cdots, B_t\not=A_1$. Applying Second Main Theorem for the non-constant map $f\colon\mathbf{K}\rightarrow H_1\cong\mathbb{P}^1(\mathbf{K})$ and the set of at least two  points $\{A_1,B_1,\dots,B_t\}$, one has
\[
T_f(r)\leq\sum_{i=1}^tN_f(r,B_i)+N_f(r,A_1)+O(1),
\]
which implies $T_f(r)\leq N_f(r,L)+N_f(r,C)+O(1)$, or equivalently 
\[
m_f(r,L)+m_f(r,C)\leq d\,T_f(r)+O(1).
\]
Hence 
\begin{align}
\label{degenerate case}
m_f(L)+\dfrac{m_f(r,C)}{d}&
\leq
\bigg(1-\dfrac{1}{d}\bigg)\, m_f(r,L)
+
\dfrac{1}{d}\big(m_f(r,L)+m_f(r,C)\big)\notag\\
&\leq
\bigg(1-\dfrac{1}{d}\bigg)\,
 T_f(r)+T_f(r)+O(1)\notag\\
 &= \bigg(2-\dfrac{1}{d}\bigg)\,T_f(r)+O(1).
 \end{align}
 Combining \eqref{non-degenerate case 1}, \eqref{non-degenerate case 2}, \eqref{degenerate case}, the proof is finished.
\end{proof}

A variety $X$ over $\mathbf{K}$ is said to be Brody $\mathbf{K}$-hyperbolic if all analytic function $f\colon \mathbf{K}\rightarrow X$ is constant. Theorem~\ref{one line one curve} is a quantitative generalization of a result by An,
Wang, and Wong \cite{An-Wang-Wong2008}.
\begin{cor}
Let $L$ be a line and $C$ be a nonsingular curve of degree $d\geq 3$ in $\mathbb{P}^2(\mathbf{K})$. If $L,C$ intersect transversally and $L$ does not pass through the point of multiplicity $d$ of $C$, then the complement $\mathbb{P}^n(\mathbf{K})\setminus(C\cup L)$ is Brody $\mathbf{K}$-hyperbolic.
\end{cor}
\begin{proof}
	If there exists a nonconstant analytic curve $f$ avoiding $L,C$, one has $m_f(r,L)=\dfrac{m_f(r,C)}{d}=T_f(r)+O(1)$.
	Suppose on the contrary that the image of $f$ is not contained in any tangent line of $C$ at the intersection point of $L$ and $C$, then it follows from Theorem~\ref{one line one curve} that $2\, T_f(r)\leq\bigg(2-\dfrac{1}{d}\bigg)\, T_f(r)+O(1)$, a contradiction.
\end{proof}
\section{Results in general cases}
\begin{thm}
	\label{general case, linearly nondegenerate}
	Let $f\colon \mathbf{K}\rightarrow\mathbb{P}^n(\mathbf{K})$ be a nonconstant analytic curve. Let $\{D_i\}_{1\leq i\leq n}$ be a family of $n$ hypersurfaces in $\mathbb{P}^n(\mathbf{K})$ intersecting transversally, with total degree $\sum_{i=1}^n\deg D_i\geq n+1$.  If  the image of $f$ is not completely contained in any $D_i$ and any tangent hyperplane of $D_i$ at the intersection points $\cap_{i=1}^nD_i$, then
	\[
	\sum_{i=1}^n\dfrac{m_f(r,D_i)}{\deg D_i}\leq \bigg(n-\dfrac{1}{n}\bigg)\,T_f(r)+O(1).
	\]
\end{thm}
\begin{proof}
Fixing the radius $r$, rearranging the indices if necessary, one may assume that 
\begin{equation}
\label{rearranging indices}
m_f(r,D_1)\geq m_f(r,D_2)\geq\cdots\geq m_f(r,D_n).
\end{equation}
If the degree of $D_n$ is at least $2$, then using Levin's arguments \cite[Theorem 10]{Levin2015}, one obtains
\[
\sum_{i=1}^n\dfrac{m_f(r,D_i)}{\deg D_i}\leq \bigg(n-1+\dfrac{1}{\deg D_n}\bigg)\,T_f(r)+O(1)
\leq
\bigg(n-\dfrac{1}{n}\bigg)\,T_f(r)+O(1).
\]

Suppose now that $D_n$ is a hyperplane. Let $i_0$ be the largest index such that $\deg D_{i_0}\geq 2$, then $1\leq i_0\leq n-1$. Let $\{P_1,\dots,P_t\}$ be the set of intersection points of $D_1,\dots, D_n$. For each $1\leq j\leq t$, let $H_j$ be the tangent hyperplane of $D_{i_0}$ at the point $P_j$, then the image of $f$ is not completely contained in any $H_j$. For each $j$ with $i_0+1\leq j\leq n$ one has
\begin{equation}
\label{max noether app-Pn}
H_1\cdots H_t=\sum_{i\not=j}R_{ij}D_{i}+R_jD_j^2,\quad i_0+1\leq j\leq n,
\end{equation}
for some homogeneous polynomials $R_{ij} (i\not=j)$ and $R_j$ with $\deg R_{ij}=t-\deg D_i$ and $\deg R_j=t-2$. If $i_0=n-1$, one deduces from the above equation that
\begin{equation}
\label{prox L<=sum of prox Hi, Pn}
m_f(r,D_n^2)\leq\sum_{i=1}^t m_f(r,H_i)+O(1)
\end{equation}
or
\begin{equation}
\label{prox C<=sum of prox Hi, Pn}
m_f(r,D_{n-1})\leq\sum_{i=1}^t m_f(r,H_i)+O(1).
\end{equation}
Using the same argument as in Theorem~\ref{one line one curve}, one gets the desired estimate.

Next, in the case where $i_0\leq n-2$, by multiplying both sides of the above $n-i_0$ equations, one receives
\[
(H_1\cdots H_t)^{n-i_0}=\prod_{j=i_0+1}^{n}\big(\sum_{i\not=j}R_{ij}D_{i}+R_jD_j^2\big).
\]
Using \eqref{rearranging indices}, one can deduce from this equality that
\[
\bigg(\dfrac{|H_1\circ f|_r\cdots |H_t\circ f|_r}{\|f\|_r^t}\bigg)^{n-i_0}
\leq
M\,\max_{1\leq k\leq n-i_0-1}\bigg\{\bigg(\dfrac{|D_n\circ f|_r}{\|f\|_r}\bigg)^{n-i_0+1},\bigg(\dfrac{|D_n\circ f|_r}{\|f\|_r}\bigg)^{n-i_0+1-k}\prod_{j=1}^{k}\bigg(\dfrac{|D_{n-j}\circ f|_r}{\|f\|_r}\bigg)^2 \bigg\},
\]
for some positive constant $M$ independent of $r$. Hence one either has

\[
m_f(r,D_n)\leq\dfrac{n-i_0}{n-i_0+1}\sum_{i=1}^t m_f(r,H_i)+O(1)
\leq \dfrac{n-1}{n}\sum_{i=1}^t m_f(r,H_i)+O(1)
\]
or
\[
(n-i_0+1-k)m_f(r,D_n)+2\sum_{j=1}^{k}m_f(r,D_{n-j})
\leq
(n-i_0)\sum_{i=1}^t m_f(r,H_i)+O(1),
\]
for some $1\leq k\leq n-i_0-1$, which implies, by using \eqref{general case, linearly nondegenerate}
\[
(n-i_0+1+k)m_f(r,D_n)\leq (n-i_0)\sum_{i=1}^t m_f(r,H_i)+O(1).
\]

Similarly as in the proof of Theorem~\ref{one line one conic} and Theorem~\ref{one line one curve}, by using Lemma~\ref{only n proximity could be not bounded}, in both cases, one always has
\[
\sum_{i=1}^n\dfrac{m_f(r,D_i)}{\deg D_i}\leq \bigg(n-\dfrac{1}{n}\bigg)\,T_f(r)+O(1).
\]
\end{proof}
\begin{cor}
Let $f\colon \mathbf{K}\rightarrow\mathbb{P}^n(\mathbf{K})$ be a nonconstant analytic curve. Let $\{D_i\}_{1\leq i\leq n}$ be a family of $n$ hypersurfaces in $\mathbb{P}^n(\mathbf{K})$ intersecting transversally, with total degree $\sum_{i=1}^n\deg D_i\geq n+1$.  If  $f$ avoids all $D_i$, then the image of $f$ must be contained in some tangent hyperplane of $D_i$ at some point of $\cap_{i=1}^nD_i$.
\end{cor}

Before closing this note, we would like to mention a result in \cite{An-Levin-Wang2011} about the Brody $\mathbf{K}$-hyperbolicity  of the complement $\mathbb{P}^n(\mathbf{K})\setminus\cup_{i=1}^nD_i$, where $\{D_i\}_{1\leq i\leq n}$ is a family of $n$ hypersurfaces in $\mathbb{P}^n(\mathbf{K})$   under some  additional assumptions generalizing those of Theorem~\ref{one line one curve}. It is interesting to seek a quantitative version of this result, namely a Second Main Theorem for nonconstant analytic curve $f\colon \mathbf{K}\rightarrow \mathbb{P}^n(\mathbf{K})$ intersecting the family $D_i$.

\section*{Acknowledgements}
The key technique in this note  is inspired by the method of employing auxiliary curves in plane geometry which is often used in high school. I would like to express my profound gratitude to my  maths teachers for their encouragements during the high school years. I want to thank Song-Yan Xie for helpful suggestions
which improved the exposition. This research is funded by University of Education, Hue University under grant number NCTB-T.24-TN.101.01.
\begin{center}
	\bibliographystyle{alpha}
	\thispagestyle{empty}
	\bibliography{references.bib}

\end{center}
\end{document}